\documentclass[12pt,a4paper]{amsart}
\usepackage{a4wide}
\usepackage{fancyhdr}
\usepackage{amsmath,amssymb,amsthm}
\usepackage{enumerate}
\usepackage{parskip}
\usepackage{graphicx}
\usepackage{verbatim}

\newtheorem{corollary}{Corollary}[section]

\newtheorem{lemma}{Lemma}[section]
\newtheorem{proposition}{Proposition}[section]
\newtheorem{theorem}{Theorem}[section]

\makeatletter
\@addtoreset{equation}{section}

\newcommand{\bes}{\begin{displaymath}}
\newcommand{\ees}{\end{displaymath}}
\newcommand{\be}{\begin{equation}}
\newcommand{\ee}{\end{equation}}
\newcommand{\ba}{\begin{eqnarray}}
\newcommand{\ea}{\end{eqnarray}}
\newcommand{\bas}{\begin{eqnarray*}}
\newcommand{\eas}{\end{eqnarray*}}
\newcommand{\@Bbb}[1]{\ensuremath{\Bbb #1}}

\newcommand{\B}{{\@Bbb B}}
\newcommand{\C}{{\@Bbb C}}
\newcommand{\E}{{\@Bbb E}}
\newcommand{\F}{{\@Bbb F}}
\newcommand{\G}{{\@Bbb G}}
\renewcommand{\P}{{\@Bbb P}}

\newcommand{\Q}{{\@Bbb Q}}
\newcommand{\bQ}{{\@Bbb Q}}
\newcommand{\N}{{\@Bbb N}}
\newcommand{\R}{{\@Bbb R}}
\newcommand{\T}{{\@Bbb T}}
\newcommand{\bbR}{{\@Bbb R}}
\newcommand{\W}{{\@Bbb W}}
\newcommand{\Z}{{\@Bbb Z}}
\newcommand{\bbZ}{{\@Bbb Z}}

\newcommand{\@s}[1]{\ensuremath{\mathcal #1}}
\newcommand{\cA}{\@s A}
\newcommand{\cB}{\@s B}
\newcommand{\cC}{\@s C}
\newcommand{\cD}{\@s D}
\newcommand{\cE}{\@s E}
\newcommand{\cF}{\@s F}
\newcommand{\cG}{\@s G}
\newcommand{\cH}{\@s H}
\newcommand{\cI}{\@s I}
\newcommand{\cJ}{\@s J}
\newcommand{\cal}{\mathcal}
\newcommand{\cK}{\@s K}
\newcommand{\cL}{\@s L}
\newcommand{\cN}{\@s N}
\newcommand{\cM}{\@s M}
\newcommand{\cO}{\@s O}
\newcommand{\cP}{\@s P}
\newcommand{\cQ}{\@s Q}
\newcommand{\cR}{\@s R}
\newcommand{\cS}{\@s S}
\newcommand{\cT}{\@s T}
\newcommand{\cU}{\@s U}
\newcommand{\cV}{\@s V}
\newcommand{\cW}{\@s W}
\newcommand{\cX}{\@s X}
\newcommand{\cY}{\@s Y}
\newcommand{\cZ}{\@s Z}

\def\d{{\rm d}}

\def\i{{\rm i}}

\newcommand{\@bm}[1]{\ensuremath{\mathbf #1}}
\newcommand{\bma}{\@bm a}\newcommand{\bmA}{\@bm A}
\newcommand{\bmb}{\@bm b}\newcommand{\bmB}{\@bm B}
\newcommand{\bmc}{\@bm c}\newcommand{\bmC}{\@bm C}
\newcommand{\bmd}{\@bm d}\newcommand{\bmD}{\@bm D}
\newcommand{\bme}{\@bm e}
\newcommand{\bmf}{\@bm f}\newcommand{\bmF}{\@bm F}
\newcommand{\bmg}{\@bm g}\newcommand{\bmG}{\@bm G}
\newcommand{\bmh}{\@bm h}\newcommand{\bmH}{\@bm H}
\newcommand{\bmi}{\@bm i}\newcommand{\bmI}{\@bm I}
\newcommand{\bmj}{\@bm j}
\newcommand{\bmk}{\@bm k}\newcommand{\bmK}{\@bm K}
\newcommand{\bml}{\@bm l}
\newcommand{\bmm}{\@bm m}\newcommand{\bmM}{\@bm M}
\newcommand{\bmn}{\@bm n}
\newcommand{\bmo}{\@bm o}
\newcommand{\bmp}{\@bm p}
\newcommand{\bmq}{\@bm q}\newcommand{\bmQ}{\@bm Q}
\newcommand{\bmr}{\@bm r}
\newcommand{\bms}{\@bm s}\newcommand{\bmS}{\@bm S}
\newcommand{\bmt}{\@bm t}
\newcommand{\bmu}{\@bm u}\newcommand{\bmU}{\@bm U}
\newcommand{\bmw}{\@bm w}\newcommand{\bmW}{\@bm W}
\newcommand{\bmv}{\@bm v}\newcommand{\bmV}{\@bm V}
\newcommand{\bmx}{\@bm x}\newcommand{\bmX}{\@bm X}\newcommand{\bx}{\@bm x}
\newcommand{\bmy}{\@bm y}\newcommand{\bmY}{\@bm Y}\newcommand{\by}{\@bm y}
\newcommand{\bmz}{\@bm z}\newcommand{\bmZ}{\@bm Z}

\newcommand{\bmzero}{\@bm 0}

\newcommand{\Lin}{\mathop{\rm Lin}}

\newcommand{\@g}[1]{\ensuremath{\mathfrak #1}}
\newcommand{\gA}{\@g A}
\newcommand{\gD}{\@g D}
\newcommand{\gJ}{\@g J}
\newcommand{\gF}{\@g F}
\newcommand{\gM}{\@g M}
\newcommand{\gR}{\@g R}

\newcommand{\diam}{\mathop{\rm diam}}

\newcommand{\commentout}[1]{{}}

\makeatother

\title[Accessibility of convex bodies]{
\large The accessibility of convex bodies \\  
and derandomization  \\
of the hit and run algorithm 
}
\author{Beno\^\i{}t Collins}
\address{B.C: D\'epartement de Math\'ematique et Statistique, Universit\'e d'Ottawa,
585 King Edward, Ottawa, ON, K1N6N5 Canada,
and Department of Mathematics,
Kyoto University, Japan
and
CNRS, Institut Camille Jordan Universit\'e  Lyon 1,
France}
\email{bcollins@uottawa.ca}

\author{Termeh Kousha}
\address{T.K: D\'epartement de Math\'ematique et Statistique, Universit\'e d'Ottawa,
585 King Edward, Ottawa, ON, K1N6N5 Canada}
\email{tkousha@uottawa.ca}

\author{Rafa\l{} Kulik}
\address{R.K.: D\'epartement de Math\'ematique et Statistique, Universit\'e d'Ottawa,
585 King Edward, Ottawa, ON, K1N6N5 Canada}
\email{rkulik@uottawa.ca}

\author{Tomasz Szarek}
\address{T.S.: Department of Mathematics, University of Gda{\'n}sk,
ul. Wita Stwosza 57, 80-952 Gda{\'n}sk, Poland}
\email{szarek@intertele.pl}

\author{Karol {\.Z}yczkowski}
\address{K.{\.Z}: Institute of Physics, Jagiellonian University, Cracow,
                                and
  Center for Theoretical Physics, Polish Academy of Sciences, Warsaw}
\email{ karol@tatry.if.uj.edu.pl}
\begin{document}

\date\today

\maketitle

\begin{abstract}
We introduce the concept of accessibility and prove that any convex body $X$
 in $\mathbb R^d$ is accessible with relevant constants depending on $d$ only. 
This property leads to a new algorithm which may be considered as a natural 
derandomization of the {\sl hit and run} algorithm applied to 
generate a sequence of random points covering $X$ uniformly.
We prove stability of the Markov chain generated by the proposed algorithm
and provide its rate of convergence.
\end{abstract}

\section{Introduction}

We are concerned with the accessibility of convex bodies in $\mathbb R^d$. Generally speaking, a convex set $X$ is accessible if there exist two integers $k$ and $l$ and a set of vectors $\{{e}_1,\ldots, {e}_l\}$ such that we may get from an arbitrary point $x\in X$ to a fixed point $y\in X$ in $k$ steps walking in the directions given by the vectors and not 
leaving 
the set $X$. Our main result - Theorem \ref{Main_THM} presented in Section 2 -- says that any convex body in $\mathbb R^d$ is accessible. We also provide universal constants $k, l$, which depend only on the dimension $d$.

In Sections 3 and 4  we propose to consider an algorithm which may be treated as a natural derandomization of the {\sl hit and run} algorithm (see \cite{DFK, Lovasz, LV1, LV2, Vempala}). Recall that the latter algorithm picks a random point along a random direction through the current point. Random directions are chosen uniformly on the sphere  $S^{d-1} \subset \mathbb R^d$.
 In the algorithm proposed in this work we also select the line through the current point randomly
 but their directions are restricted to the set of given vectors  $\{{e}_j \}_{j=1}^l$. 
If the vectors are such that the set is accessible the algorithm is exponentially convergent to the Lebesgue distribution. We provide its rate of convergence. There is no surprise that the algorithm is generically slower than the hit and run algorithm but the fact that it is also convergent at an exponential rate seems to be interesting 
per se.

Finally, in Section 5 we apply the algorithm to concrete convex bodies appearing in quantum information theory and statistical physics (quantum states, stochastic matrices and bistochastic matrices).
\section{Accessibility of convex bodies}

To characterize the set $X$ and the set of vectors ${\mathbf e}=\{e_1,\dots e_l\}$, $l\ge d$,
 we will need the notion of
accessibility in $k$ steps illustrated in Fig. \ref{fig:acc1}.

{\bf Definition}.
A compact set $X\subset\mathbb R^d$ is called {\bf accessible} in $k$  steps
with respect to $l$ vectors of ${\mathbf e}$,
if there exists $x_* \in {\rm int} X$ such that
from any point $x\in X$ one can reach $x_*$
in not more than $k$ moves along the basis vectors.
Thus there exist some sets $\{i_1,\ldots, i_k\}\subset\{1,\ldots, l\}$
and $\{\lambda_1,\ldots, \lambda_k\}$, $\lambda_1,\ldots,\lambda_k\in\mathbb R$ such that
$$
x+\sum_{j=1}^m \lambda_j \mathbf e_{i_j}\in {\rm int} X\qquad\text{for any $m\le k$}\quad\text{and}\quad
 x+\sum_{j=1}^k \lambda_j \mathbf e_{i_j}=x_*.
$$

\begin{figure}[ht]
\centering
\scalebox{0.85}{\includegraphics{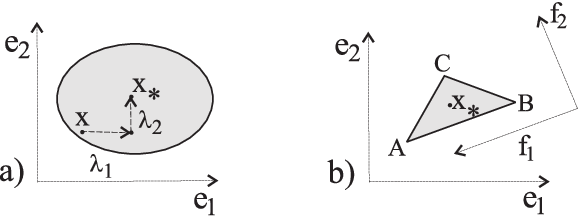}}
\caption{Examples of accessibility in $2D$:
a) an ellipse is $2$--accessible with respect to any
orthogonal basis: any point $x$ can be transformed into
a selected point $x_*$ in two moves along the basis vectors;
b) the triangle $ABC$ is not accessible with respect to basis $(e_1,e_2)$, 
as starting from the corner $A$ one cannot move
along the basis vectors,
but this triangle is $2$--accessible with respect to the basis $(f_1,f_2)$.}
 \label{fig:acc1}
\end{figure}

If a set $X$ is accessible in $k$ steps with respect to a fixed
set of vectors of ${\mathbf e}$ we will briefly call it {\sl $k$--accessible}.
We are aimed at proving that any convex body in $\mathbb R^d$ is accessible with proper $k$ and $l$ depending on $d$ only.

At the very beginning we prove the following lemma.

\begin{lemma} Assume that a convex body $X\subset\mathbb R^d$ is such that $\diam X\le R$ and there is a $r$--ball $B$ such that
$
B\subset X.
$
Then $X$ is  accessible in $d+1$ steps with respect to some $\{{e}_1,\ldots, {e}_l\}$, where $l\le (1+2R/r)^d+d$.
\end{lemma}

\begin{proof} Let $x_*$ denote the center of the ball $B$. Let $B(R)$ denote the ball centred at $0$ with diameter $R$ and let $S$ be its sphere, i.e.
$$
S=\{x\in\mathbb R^d: \|x\|=R\}.
$$
Choose an open cover $\mathcal U=\{ U_1,\ldots, U_q\}$ of the sphere $S$ such that $\diam U_i\le r/2.$  We may assume that $q\le (1+2R/r)^d$, by Lemma 4.10 in \cite{GP}.
Choose $x_i\in U_i$ and let ${e}_i=[0, x_i]$ for $i=1,\ldots, q$. Finally set ${e}_i=e_{i-q}$ for $i\in\{q+1,\ldots, q+d\}$, where $\{e_1,\ldots, e_d\}$ is the standard basis in $\mathbb R^d$. 

Consider now the sets of vectors $\vec{U}_i=\{[0, x]: x\in U_i\}$ for $i=1,\ldots, q$. 
Choose 
$x\in X$. The set
$$
\{x+\|x_*-x\|R^{-1}\cdot\vec{U}_i\}_{i=1}^q
$$
is a cover of the sphere $\tilde S=\{y\in\mathbb R^d: \|y-x\|=\|x_*-x\|\}$. Obviously $x_*\in \tilde S$. Since $\|x_*-x\|\le R$, the diameter of all sets of the cover is less than $r/2$. Thus there is $i\in\{1,\ldots, q\}$ such that
$$
x+\|x_*-x\|R^{-1}\cdot\vec{U}_i\subset B(x_*, r)\cap \tilde S
$$
and consequently
$$
x+\|x_*-x\|R^{-1}\cdot {e}_i\in B(x_*, r).
$$
Now after at most $d$--steps along the natural basis $\{e_1,\ldots, e_d\}$ we reach the point $x_*$. The proof is complete.
\end{proof}

A map $\varphi: \mathbb R^d\to\mathbb R^d$ is called an affine map if it is of the form
$$
\varphi (x)=Ax+b \qquad\text{for $x\in\mathbb R^d$},
$$
where $A:\mathbb R^d\to\mathbb R^d$ is linear and $b\in\mathbb R^d$ is a constant vector. The affine map $\varphi$ is nonsingular iff ${\rm det}\, A\neq 0$. An {\bf ellipsoid} in $\mathbb R^d$ is the image of a unit ball $B\subset\mathbb R^d$ under some nonsingular affine map.

To proceed we shall make
use of the following famous result of F. John (see \cite{FJ}). 

\begin{theorem}\label{Thm1.10.02.15}

 Let $X\subset \mathbb R^d$
be a convex body. Then there is an ellipsoid $E$ (called the
John ellipsoid which turns out to be the ellipsoid of maximal volume
contained in $X$) so that if $c$ is the center of $E$ then the inclusions
$$
E\subset X\subset c+d(E-c)
$$
hold.

\end{theorem}
The following lemma is obvious.

\begin{lemma}\label{Lem1.10.02.15}
If a convex body $X\subset \mathbb R^d$ is accessible in $k$ steps with respect to $\{{e}_1,\ldots, {e}_l\}$, then for any nonsingular affine map $\varphi: \mathbb R^d\to\mathbb R^d$ the set $\varphi (X)$ is also $k$--accessible with respect to $\{\varphi({e}_1),\ldots, \varphi({e}_l)\}$.
\end{lemma}
\begin{proof} Let $x_*\in  {\rm int}\,X$ be such that for any $x\in X$ we have
$$
x+\sum_{j=1}^m \lambda_j e_{i_j}\in {\rm int}\, X\qquad\text{for any $m\le k$}\quad\text{and}\quad
 x+\sum_{j=1}^k \lambda_j e_{i_j}=x_*
 $$
 for some
 $\{i_1,\ldots, i_k\}\subset\{1,\ldots, l\}$
and $\{\lambda_1,\ldots, \lambda_k\}$, $\lambda_1,\ldots,\lambda_k\in\mathbb R$.

Let $\tilde x_*=\varphi (x_*)$. We easily check that for any $y\in\varphi (X)$, $y=\varphi (x)$, we have
$$
y+\sum_{j=1}^m \lambda_j \varphi(e_{i_j})\in \varphi({\rm int}\, X)\qquad\text{for any $m\le k$}\quad\text{and}\quad
 y+\sum_{j=1}^k \lambda_j \varphi(e_{i_j})=\tilde x_*.
 $$
\end{proof}

The main result of this section is the following theorem

\begin{theorem}\label{Main_THM}
Assume that $X\subset\mathbb R^d$
is a convex body. Then $X$ is $d+1$ accessible with respect to some $\{{e}_1,\ldots, {e}_l\}$, where $l\le (2d+1)^{d}+d$.
\end{theorem}

\begin{proof} From Theorem \ref{Thm1.10.02.15} it follows that
$$
\varphi (B)\subset X\subset c+d(\varphi (B)-c)
$$
for some nonsingular affine map $\varphi=A\cdot+b$ and a unit ball $B$. Thus
$$
B\subset \varphi^{-1}(X)\subset \varphi^{-1}(c+d(\varphi (B)-c))=d B +(d-1)A^{-1}(b-c).
$$
Hence $\diam \varphi^{-1}(X)\le 2d$ and a ball $B$ with radius $1$ is contained in $\varphi^{-1}(X)$. From Lemma 2.2 it follows that $\varphi^{-1}(X)$ is  accessible in $d+1$ steps with some $\{{\hat e}_1,\ldots, {\hat e}_l\}$, where $l\le (2d+1)^{d}+d$. Since $\varphi^{-1}$ is a nonsingular affine map, our hypothesis follows from Lemma \ref{Lem1.10.02.15}.
\end{proof}

\section{Convergence theorem, general setup}\label{section:problem}
Let $(X, \mathcal A)$ and $(I, \mathcal B)$ be two measurable spaces and let $\mu$, $\nu$ be two probability measures
 on $X$ and $I$, respectively.

We shall assume that for any $i\in I$ we have a transition 
kernel 
$T_i: X\times \mathcal A\to [0, 1]$,
i.e. $T_i(x, \cdot)$ is a probability measure for any $x\in X$ and for any $A\in\mathcal A$ the
 function $T_i(\cdot, A): X\to [0, 1]$ is measurable. Additionally, we assume that $\mu$ is invariant with respect
 to $T_i$ for any $i\in I$:
$$
\mu (A)=\int_XT_i(x, A)\mu (dx)\qquad\text{for all $A\in\mathcal A$.}
$$
Now, if we assume that for any $A\in\mathcal A$ the function $T_{\cdot}(\cdot, A): I\times X\to[0, 1]$ is
 $\mathcal B\otimes_{\sigma}\mathcal A$--measurable, then it follows from the Fubini theorem that the measure 
 $\mu$ is invariant with respect to the operator $Q$ of the form
$$
Q\hat\mu(\cdot)=\int_X\int_I T_i(x,\cdot)\hat{\mu}(dx)\nu(di).
$$

By $\mathcal M$ and $\mathcal M_1$ we shall denote the set of all Borel measures and all probability Borel measures on $X$, respectively.

By $\|\cdot\|_{TV}$ we denote the {\bf total variation norm}, i.e., if $\hat{\mu}\in\mathcal M-\mathcal M$,
then $\|\mu\|_{TV}:={\hat\mu}^+(X)+\hat{\mu}^-(X)$,
where $\hat{\mu}=\hat{\mu}^+-\hat{\mu}^-$ is the Jordan decomposition of the signed measure $\hat{\mu}$.

We start with the following version of Doeblin's theorem \cite{Doeblin},
which provides sufficient conditions 
for exponential convergence rates of the transition operator $Q$.

\begin{proposition}\label{p1_10.06.13} Assume that there
 exist $\theta\in (0, 1)$, $M\in\mathbb N$ and a
 measure $\nu\in\mathcal M_1$ such that for any measurable set $A$
$$
Q^M(x, A)\ge \theta\nu(A)\qquad\text{for any $x\in X$}.
$$
Then there exists a unique invariant measure $\mu_*\in\mathcal M_1$ such that
\begin{equation}\label{e3}
\|Q^n\mu-\mu_*\|_{TV}\ \le \ C\alpha^n\qquad\text{for all $\mu\in\mathcal M_1$
and $n\ge 1$},
\end{equation}
with the convergence rate $\alpha=(1-\theta)^{1/M}$ and
prefactor $C=2(1-\theta)^{-1}$.

\end{proposition}

Assume now that $X$ is a bounded metric space and let $\varphi :X\to\mathbb R$ be
a Lipschitz function such that $\int_X \varphi(x)\mu (dx)=0$.
 Then we have (see Theorem 17.5.4 in \cite{MT}):

\begin{proposition}\label{prop11.06.13} Let $(\Phi_n)$ be the Markov chain corresponding to the transition operator $Q$.
Under the hypothesis of Proposition \ref{p1_10.06.13}, we have the Central Limit Theorem (CLT),
$$
\frac{\sum_{i=1}^n\varphi (\Phi_i)}{\sqrt{n}}\Longrightarrow W, \quad\text{as $n\to+\infty$,}
$$
where $W$ is a random variable with normal distribution $\mathcal N(0, D)$ for some $D\ge 0$ and the
convergence is understood in law. Moreover, we have the  Law of the Iterated Logarithm (LIL),
$$
\limsup_{n\to +\infty} \frac{\sum_{i=1}^n\varphi (\Phi_i)}{\sqrt{2n\log\log n}}=D
$$
with probability $1$. Of course the above implies that also
$$
\liminf_{t\to +\infty} \frac{\sum_{i=1}^n\varphi (\Phi_i)}{\sqrt{2n\log\log n}}=-D
$$
with probability $1$.
\end{proposition}

\section{The algorithm and its convergence rate}\label{section:conv_rates}

We first describe the algorithm which is a natural derandomization of the hit and run algorithm. 

\subsection{Description of the algorithm:}
We consider a compact set $X \subset {\mathbb R}^d$. 
In a preliminary step let us choose a set of normed vectors
$e=\{e_1,\dots e_l\}$ in ${\mathbb R}^d$ for $l\ge d$ such that $\Lin \{e_1,\dots e_l\}=\mathbb R^d$.
To generate a sequence
of random points in $X$ repeat the following steps of the
algorithm, illustrated in Fig. \ref{fig:acc2}.

\begin{enumerate}
\item
 Choose an arbitrary starting point $x_0 \in X$,
\item
Draw randomly a vector $e_i$, where the direction $i$
   is chosen with a uniform distribution among $(1,\dots, l)$.
\item
Find boundary points  $x_1^{\rm min}, x_1^{\rm max} \in \partial X$
along the direction $e_i$:
there exist positive numbers $a,b$ such that
$x_1^{\rm min}=x_0-a e_i$ and $ x_1^{\rm max}+be_i$.
\item
Select a point $x_1$  randomly with respect to the
    uniform measure in the interval $[x_1^{\rm min}, x_1^{\rm max}]$.
\item
Repeat the steps (ii)-(iv) to find subsequent random points
  $x_2,x_3,\dots$.
\end{enumerate}

This algorithm is very close to slice sampling. 
The main difference is that slice sampling proceeds recursively. It assumes one
is able to simulate a uniform measure on an slice of codimension $1$ (obtained by intersecting with an $n-1$ dimensional
affine hyper plane).

\begin{figure}[ht]
\centering
\scalebox{0.85}{\includegraphics{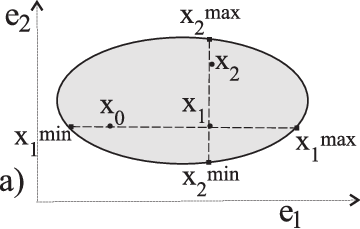}}
\caption{Algorithm to generate random points
uniformly in a compact set $X$: one starts with any interior point $x_0 \in X$,
picks randomly a direction (in this case $e_1$)
finds two boundary points and draws $x_1$
randomly in the interval $[x_1^{\rm min}, x_1^{\rm max}]$.
In the next step one chooses a next direction (here $e_2$),
finds both boundary points and draws  $x_2$
randomly in the interval $[x_2^{\rm min}, x_2^{\rm max}]$.}
 \label{fig:acc2}
\end{figure}

\subsection{Convergence rate with respect to the Lebesgue measure}
Let $X$ be a compact subset of $\mathbb R^d$
with a nonempty interior.
Let a point $x_*$ in the interior of $X$ be given.
We introduce two positive constants, $r$ and $R$,
such that  $B(x_*, r)\subset X\subset B(x_*, R)$,
where $B(x_*,r)$ denotes a closed ball in $\mathbb R^d$.
Let ${\mathbf e}=\{ e_1,\ldots, e_l \}$
such that $\Lin\{e_1,\ldots, e_l \}=\mathbb R^d$ be given.
We consider the Markov chain $\mathbf\Phi=(\Phi_n)_{n\ge 1}$ corresponding to the algorithm
illustrated in Fig. \ref{fig:acc2}, which is described by
 the following transition function
\begin{equation}
\label{e1}
T (x, A )\ = \ \frac{1}{l} \; \sum_{i=1}^l \nu_{x, i}(A),
\end{equation}
where $\nu_{x, i}$ is the measure uniformly
distributed over the set
$\{y\in X: y-x=t e_i\,\,\text{for some $t\in\mathbb R$}\}$.

We are in a position to formulate the following theorem.

\begin{theorem}\label{convergence-estimate}
 If $X\subset\mathbb R^d$
 is $k$--accessible with respect to the basis ${\mathbf e}=\{e_1,\ldots, e_l\}$, 
$x_*\in X$
and  $B(x_*, r)\subset X\subset B(x_*, R)$ for
some $r, R>0$, then the chain $\mathbf\Phi$
corresponding to the transition function $T$
given by (\ref{e1}) satisfies the hypothesis of Proposition 3.1  with

\begin{equation}
\label{e2}
M=k+d\,\,{\rm and}\,\,\,\,\theta=b_d \; l^{-k-d} (r/R)^{k+d},
\end{equation}
where $b_d=\pi^{d/2}/\Gamma(d/2+1)$
denotes the volume of a unit ball in $\mathbb R^d$.

\end{theorem}

\begin{proof} First, there is no restriction in assuming that $r\le 1$. Fix $x\in X$ and
let $\{i_1,\ldots, i_k\}$ and $\{\lambda_1,\ldots, \lambda_k\}$ be
given according to $k$--accessibility of $X$. We are going to derive
 lower bounds for subdensities $f_m$ of $T^m(x,A)$ for $m=1,\ldots, k$
defined on suitable spaces.
Firstly, let $\tilde X_1:=\{x+t_1 e_{i_1}: t_1\in\mathbb R\}\cap X$.
By the fact that $\tilde{X_1}$ contains a non-degenerate interval
its one dimensional  measure  $\mathcal L_1$ is positive. Set $f_1(y)=1/R$ for $y\in\tilde X_1$. Finally, we easily check
 that for any measurable set $A$ the following inequality holds
$$
T (x, A)\ \ge \ (1/l)\int_{ A\cap \tilde X_1} f_1 \d\mathcal L_1.
$$
Define $\tilde X_2=\bigcup_{t_2\in\mathbb R} (\tilde X_{1}+t_2 e_{i_2})\cap X$. 
The set $\tilde X_2\neq 0$ and its $t_2$--dimensional Lebesgue
 measure $\mathcal L_{t_2}$ is positive, where $t_2=\dim \Lin\{e_{i_1}, e_{i_2}\}$,
by $k$--accessibility of $X$. Set $f_2(y):=(1/R)^2$ for $y\in\tilde X_2$ and observe that
$$
T^2 (x, A )\ \ge \ (1/l)^2\int_{A \cap \tilde X_2} f_2 \d\mathcal L_{t_2}.
$$

By induction we define the sets $\tilde X_3,\ldots, \tilde X_m$.
If we have done it for $m<k$, we may do it for $m+1$ as well. Namely, we set
$\tilde X_m:=\bigcup_{t_m\in\mathbb R} (\tilde X_{m-1}+t_m e_{i_m})\cap X$.
Then the set $\tilde X_m$ has positive $t_m$--dimensional Lebesgue measure,
where $t_m=\dim\Lin\{e_{i_1},\ldots, e_{i_m}\}$. Moreover,
$$
T^m (x, A ) \ \ge \ (1/l)^m\int_{ A \cap \tilde X_m} f_m \d\mathcal L_{t_m},
$$
where $f_m(y):=(1/R)^m$ for $y\in \tilde X_m$.
 In this way we obtain the constant $t_k=\dim\Lin\{e_{i_1},\ldots, e_{i_k}\}$ and the set $\tilde X_k\ni x_0$.
Set
$$
e_{i_{k+1}}=e_{q_1},\ldots, e_{i_{k+d}}= e_{q_d},
$$
where $q_1,\ldots, q_d\in\{1,\ldots, l\}$ are such that $\Lin\{e_{q_1},\ldots, e_{q_d}\}=\mathbb R^d$.

Now we repeat the procedure $d$ more times.
Then $\tilde X_{k+d}\supset B(x_0, r)$ and $t_{k+d}=d$ and
we finally obtain
$$
T^{k+d} (x, A)\ \ge \ (1/Rl)^{k+d}\mathcal L_{d}( A \cap B(x_0, r)).
$$

Since, $B(x_0, r)\subset\tilde X_{k+d}$ we obtain
$$
T^{k+d} (x, A)\ \ge  \ (1/Rl)^{k+d}b_d r^d\nu(A) \ \ge \ b_d l^{-k-d}(r/R)^{k+d}\nu(A),
$$
where $\nu(A)=\mathcal L_{d}(A \cap B(x_0, r))(\mathcal L_{d}(B(x_0, r)))^{-1}$.

This completes the proof.
\end{proof}

From  Proposition \ref{prop11.06.13} it follows that

\begin{corollary}
Let $(\Phi_n)_{n\ge 1}$ be the Markov chain corresponding to the transition function $T$ and let $\varphi$ be an arbitrary
Lipschitz function on $X$ such that $\int_X\varphi d\mathcal L_d=0$. Then $(\varphi(\Phi_n))_{n\ge 1}$ satisfies the  CLT and LIL.

\end{corollary}

\subsection{Case where the density is not uniform}

In this section we consider the case where the measure on a compact subset $K$ of $\mathbb{R}^n$ is not
just the restriction of the Lebesgue measure, but a more general measure.
We assume that the measure has a continuous strictly positive density with respect to the Lebesgue measure, and
denote this density by $f$.

In this case, we consider the subset $\tilde K$ of $\mathbb{R}^n\times \mathbb{R}^+$ given by
$$\tilde K = \{(x,y): x\in K,   0<y\leq f(x)\}.$$
This is clearly a compact set, and we can prove readily the following Proposition, 
from which we derive our estimates.

\begin{proposition}
Assume that we are under the hypotheses of Theorem \ref{convergence-estimate}
for the set $\tilde K$.
Then the law of the first component $x$ of $(x,y)$ converges towards the probability measure
that we want to simulate $f(x)dx$ at the speed given by Theorem \ref{convergence-estimate}.
\end{proposition}

\begin{proof}
This is a direct application of Fubini's theorem.
\end{proof}

\section{Applications to quantum information theory and statistical physics}
\label{section:application}

In several problems of statistical and quantum physics
one works with states defined on an $N$ dimensional space.
The corresponding sets of states form compact convex sets
in $\mathbb R^d$, were $d=d(N)$. The same is true for
the set of linear discrete transformations acting on them.
In the classical case one uses {\sl stochastic} and
 {\sl bistochastic matrices},
which send the set of probability simplex into itself,
while in the quantum case one deals with {\sl quantum operations},
formally defined as completely positive, trace preserving maps.

In several concrete applications, related e.g. to the theory of
quantum information,
 one considers often various convex subsets of the above sets,
and is interested to analyze properties of their typical
elements. To this end it is important to develop an
efficient algorithm to generate a sample of random points
according to the flat measure in a given convex set $X$.
In some cases there exist such algorithms dedicated
to a given set: For instance, procedures to generate
random quantum states were studied in \cite{Br96,Ha98,ZS01},
while other contributions deal with
random subnormalized states \cite{CSZ07},
random bistochastic matrices \cite{CSBZ09}
and random quantum operations \cite {BCSZ09}.

However, the procedures mentioned above
are dedicated to a particular problem and
cannot be easily adopted to other convex sets.
On the other hand, the sampling algorithm developed in this work
is universal, as it allows one to generate
sequences of random points distributed uniformly
in an arbitrary convex body $X \subset {\mathbb R}^d$.
We are going to characterize $X$ by the
radius $R$ of the minimal outscribed sphere,
the radius $r$ of the maximal inscribed sphere
and by the barycenter $x_*$.

The rest of the paper is devoted to providing explicit estimates of
the convergence rate of the algorithm to generate random points in $X$,
which depends on the dimensionality $d$, the ratio $\mu=r/R$
and on the accessibility parameter $k$.

\subsection{Balls, cubes and simplexes in $\mathbb R^d$.}

For balls and cubes 
in ${\mathbb R}^d$
it is not difficult to generate random points
according to the uniform measure,
so we will not advocate
to use the above algorithm for this purpose.
However it is illuminating to compare estimations
for the parameters determining the convergence rate
according to Eq. (\ref{e2}).

\subsubsection{The Euclidean ball}
 For a unit {\bf ball} $B^d$ both radii coincide, $R=r$,
so their ratio $\mu=r/R$ is equal to unity.
Since for any choice of the basis ${\mathbf e}$
the ball is $d$ accessible,
 estimation (\ref{e2}) gives
$M=2d$ and $\theta=b_d\; d^{-2d}$
where $b_d$ denotes the volume of a unit $d$--ball.
This implies the convergence rate
of our algorithm applied to a $d$--ball,
$$\alpha=(1-\theta)^{1/M}=(1-b_d d^{-2d})^{1/2d}.$$
 
\subsubsection{The unit cube}
For an unit {\bf cube} $C^d$ the inscribed radius $r=1/2$,
and outscribed radius $R=\frac{1}{2}\sqrt{d}$
so the ratio  reads $\mu=r/R=1/\sqrt{d}$. If the basis
$\mathbf e$ is determined by the sides of the cube
then $C^d$ is $d$--accessible. This implies
$M=2d$ and $\theta=b_d d^{-3d}$
and yields the convergence rate
 $$\alpha=(1-b_d d^{-3d})^{1/2d}.$$
 
\subsubsection{The simplex}\label{simplex}
For an $N$--{\bf simplex} $\Delta_{N}$ embedded in
${\mathbb R}^d$ with $d=N-1$ we have $R=\sqrt{(N-1)/N}$
and $r=1/\sqrt{N(N-1)}$ so that $\mu=1/(N-1)=d^{-1}$.
Note that the simplex $\Delta_{N}$
describes the set of classical states --
$N$-point probability distributions.

For a $d$--simplex we can find a basis ${\mathbf e}$
such that the set $\Delta_{d+1}$ is $d$--accessible.
This is the case if the first vector $e_1$ is parallel
to a side of the simplex,
$e_2$ and $e_1$ span the plane parallel to a face
of $\Delta_{d+1}$, while adding an additional vector
$e_n$ spans a hyperplane containing an $n$--face of the
simplex. For this choice of the basis one obtains therefore
$M=2d$ and $\theta=b_d\; d^{-4d}$,
which implies
$$\alpha= (1-b_d d^{-4d})^{1/2d}$$
in our algorithm.

\subsection{Quantum states} The set $\Omega_N$ of density matrices
(Hermitian and positive operators, $\rho^{\dagger}=\rho \ge 0$,
normalized by the trace condition Tr$\rho=1$)
of size $N$ has the dimension $d=N^2-1$.
The radius of the out-sphere, equal to the Hilbert--Schmidt distance
between a pure state diag$(1,0,\dots,0)$
and the maximally mixed state $\rho_*={\mathbb I}/N$,
reads $R=\sqrt{(N-1)/N}$.
The radius of the inscribed sphere given by the distance
between $\rho_*$ and the center of a face,
diag$(0,1,\dots,1)/(N-1)$ is equal to
$r=1/\sqrt{N(N-1)}$ hence
 $\mu=1/(N-1) =1/(\sqrt{d+1}-1) \sim d^{-1/2}$.  

Any quantum state $\rho \in \Omega_N$ can be expressed in terms of the
generalized {\sl Bloch vector} $\tau$,
\begin{equation}
\label{bloch}
\rho= \frac{1}{N}{\mathbb I}+
\sum_{i=1}^d \tau_i \lambda_i .
\end{equation}
Here $\{ \lambda_i\}$ is a set of $d=N^2-1$ traceless
generators of the group $SU(N)$, which form an orthonormal
basis in the Hilbert--Schmidt space of operators of order $N$.
For $N=2$ one usually takes three Pauli matrices $\sigma_i$
while for $N=3$ it is convenient to use eight Gell-Mann matrices \cite{Schiff}. Since the state $\rho$ is hermitian,
the coordinates of the corresponding
Bloch vector, $\tau_i = {\rm Tr} \lambda_i \rho$, are real.
Thus the Bloch vector $\tau=(\tau_1,\dots, \tau_d)$
belongs to ${\mathbb R}^d$ and the conditions
for $\tau$ to guarantee positivity of $\rho$ are known \cite{Ki03}.
Setting some coefficients of $\tau$ to zero corresponds to a
projection onto a subspace and does not spoil positivity of $\rho$.

Thus the $d$--dimensional convex set $\Omega_N$ of quantum states
is $d$--accessible with respect to the Bloch basis
$(\lambda_1,\dots, \lambda_N)$.
For this choice of the basis one obtains therefore
$M=2d=2(N^2-1)$ and $\theta \sim b_d\; d^{-3d}$,
which implies $\alpha \sim (1-b_d d^{-3d})^{1/2d}$.
Interestingly, from the point of view of the
estimation for the convergence rate
the set $\Omega_N$ of mixed quantum states,
behaves analogously as a $d$--cube $C^d$ of dimension $d=N^2-1$.

\subsection{Stochastic matrices.} Stochastic matrices of order $N$ 
form a convex body of dimensionality $d=N(N-1)$
and play a role of classical maps, which send the simplex of $N$--point
probability vectors into itself.
Each column of a stochastic matrix $T$ consists of non-negative numbers which sum to unity,
so it forms an $N$-simplex. Thus the set of stochastic matrices is equivalent to
a Cartesian product of $N$ simplexes $\Delta_N$,
so the estimates follow from section \ref{simplex}, as each column of $T$
can be generated independently.

\subsection{Bistochastic matrices.} The set ${\cal B}_N$
 of bistochastic matrices of size $N$, called
  {\sl Birkhoff polytope} and given by convex hull of all
   permutation matrices has dimensionality $d=(N-1)^2$.
The radius of the  out-sphere of ${\cal B}_N$,
equal to the Hilbert--Schmidt distance between identity
and the uniform matrix $B_*$ containing all entries equal to $1/N$
reads $R=\sqrt{N-1}$.
The radius of the inscribed sphere given by the distance
between $B_*$ and the matrix $B_0=[NB_*-{\mathbb I}]/(N-1)$
 is equal to $r=1/\sqrt{N-1}$, which implies $\mu=1/(N-1)$.

Consider the set $\mathcal C$ of all matrices of the form
\begin{center}

$C_{i \alpha \beta \gamma}=
\left[
\begin{array}{lcccr}
c_{11}=0 & \cdots & 0 & \cdots & c_{1N}=0 \\
\cdots & c_{i \alpha}=-1 & \cdots & c_{i \beta}=1 & \cdots \\
0 & \cdots & 0 & \cdots & 0 \\
\cdots & c_{\gamma \alpha}=1 & \cdots & c_{\gamma \alpha}=-1 & \cdots \\
c_{N1}=0 & \cdots & 0 & \cdots & c_{NN}=0 \\
\end{array}
\right]$

\end{center}
for $i, \alpha, \beta, \gamma\in\{1,\ldots, N\}$.
The set $\mathcal C$ will play the role of $\mathbf e=\{e_1,\ldots, e_l\}$. Obviously $l=N^2(N-1)^2.$
It may be verified that the set ${\cal B}_N$ is $(N-1)^3$--accessible with respect to $\mathbf e$.
To see it assume that $A=[a_{i, j}]_{1\le i, j\le N}$ is a bistochastic matrix with $a_{i,\alpha}>1/N$.
Let $\varepsilon=\min\{a_{i,\alpha}-1/N, 1/N\}$. Observe that since $a_{i,\alpha}>1/N$ and the matrix is bistochastic,
there exists $\beta$ such that $a_{i, \beta}<1/N$. On the other hand, since $a_{i, \beta}<1/N$, there exists
$\gamma$ such that $a_{\gamma, \beta}>1/N$. Taking $A-\varepsilon C_{i \alpha \beta \gamma}$ we obtain the
bistochastic matrix of the form
\begin{center}
$
\left[
\begin{array}{lcccr}
a_{11} & \cdots & \cdots & \cdots & a_{1N} \\
\cdots & a_{i \alpha}-\epsilon=\frac{1}{N} & \cdots & a_{i \beta}+\epsilon>0 & \cdots \\
\cdots & \cdots & \cdots & \cdots & \cdots \\
\cdots & a_{\gamma \alpha}+\epsilon>0 & \cdots & a_{\gamma \alpha}-\epsilon>0 & \cdots \\
a_{N1} & \cdots & \cdots & \cdots & a_{NN} \\
\end{array}
\right].$
\end{center}
Repeating this procedure at most $N-1$ times (possibly with different $\beta$ and $\gamma$) we obtain a matrix
with $a_{i, \alpha}=1/N$. To obtain the matrix with all entries equal to $1/N$ we have to apply this procedure to at
most $(N-1)^2$ entries and hence follows that $k=(N-1)^3$.
Finally, we have $M=N(N-1)^2$ and $\theta=b_{(N-1)^2} (N-1)^{-4N(N-1)^2}$. Hence
$$\alpha=(1-\theta)^{1/M}=(1-b_{(N-1)^2} (N-1)^{-4N(N-1)^2})^{N(N-1)^{-2}}.$$

\section{Concluding Remarks}\label{section:remarks}

The paper was mainly devoted to introducing the concept of accessibility of convex bodies. Our main result
says that any convex body in $\mathbb R^d$ is accessible with some universal constant dependent only on $d$. But in 
the paper we also proposed a universal algorithm to generate random points
inside an arbitrary compact set $X$ in ${\mathbb R}^d$
according to the uniform measure.
Any initial probability measure $\mu$ transformed
by the corresponding Markov operator converges
exponentially to the invariant measure $\mu_*$,
uniformly in $X$. Explicit estimations for the convergence rate are derived
in terms of the ratio $\mu=r/R$ between the radii of the sphere
inscribed inside $X$ and the sphere outscribed on it
and the number $k$ determining the accessibility of the body
with respect to a given orthogonal
basis ${\mathbf e}$ in ${\mathbb R}^d$.

We hope that the algorithm presented here can be used in practice
to generate, for instance, a sample
of random quantum states. In the case of
states of a composed quantum system, one can also
generate a sequence of random states with positive partial
transpose. Sampling random states satisfying a given condition
and analyzing their statistical properties is relevant
in the research on quantum entanglement and
correlations in multi-partite quantum systems.
A standard approach of generating random points
from the entire set of quantum states with respect to the flat measure \cite{ZS01}
and checking a posteriori, whether the partial transpose of the state constructed
is positive, becomes inefficient for large dimensions,
as the relative volume of the set of PPT states becomes
exponentially small \cite{ZHSL98}.

Note that the notion of $k$--accessibility plays a crucial role in
obtaining our estimations. Running the algorithm
for the triangle $ABC$ with the basis ${\bf e}$ (see Fig. \ref{fig:acc1}b),
with respect to which it is not finitely accessible,
one would cover an open subset of the triangle (with two corners excluded).
Although this set has the full measure of the triangle
it is an open set, so the convergence will not be exponential.

In general, for any $k$--accessible set,
the lower parameter $k$ characterizing the accessibility is,
the faster convergence of the Markov chain
to the unique invariant measure $\mu_*$ one obtains.

\medskip

\section*{Acknowledgements}

BC, TK and RK were supported by NSERC discovery grants and Ontario's ERA grants.
The research of TS was supported by Polish Ministry of Science and Higher Education Grant no. IdP2011 000361 and  EC grant RAQUEL. 
He was also supported by the National Science Centre of Poland, grant number DEC-2012/07/B/ST1/03320,
 while KZ acknowledges a support by the NSC grant DEC-2011/02/A/ST1/00119.
The authors wish to thank David McDonald for many discussions.
A part of this work emanates from TK's PhD thesis.
Substantial progress was achieved during several peer visits of the authors to their
co--authors' home institutions  (and of TS and BC to Cambridge's INI). 
The hospitality of the aforementioned institutions is warmly acknowledged. 
It is a pleasure to thank R. Adamczak for constructive remarks
and for drawing our attention to \cite{Vempala}.

\end{document}